\documentclass{amsart}
\usepackage{amsthm,amssymb,amsmath,dsfont,bigdelim}
\usepackage{graphicx}
\usepackage{float}
\usepackage[arc,all]{xy}
\usepackage{longtable}
\newtheorem{thm}{Theorem}[section]
\newtheorem{cor}[thm]{Corollary}
\newtheorem{lem}[thm]{Lemma}
\newtheorem{prop}[thm]{Proposition}
\theoremstyle{definition}

\newtheorem*{theorem}{\textbf{Main Theorem}}
\newtheorem*{proof M}{\textbf{Proof of  Main Theorem}}
\theoremstyle{remark}

\numberwithin{equation}{section}

\begin{document}
\title[ Schur'theorem]{Characterizing nilpotent Lie algebras that satisfy on converse of the Schur's theorem}%

\author{afsaneh shamsaki}%
\address{School of Mathematics and Computer Science\\
Damghan University, Damghan, Iran}
\email{Shamsaki.afsaneh@yahoo.com}%
\author[P. Niroomand]{Peyman Niroomand}\thanks{Corresponding Author: Peyman Niroomand}
\email{niroomand@du.ac.ir, p$\_$niroomand@yahoo.com}
\address{School of Mathematics and Computer Science\\
Damghan University, Damghan, Iran}

\begin{abstract}
Let $ L $ be a finite dimensional nilpotent Lie algebra and $ d $ be the minimal number generators for $ L/Z(L). $ It is known that $ \dim L/Z(L)=d \dim L^{2}-t(L)$ for an integer $ t(L)\geq 0. $ 
In this paper, we classify all finite dimensional nilpotent Lie algebras $ L $ when $ t(L)\in \lbrace 0, 1, 2 \rbrace.$ We find also  a construction, which shows that there exist Lie algebras of arbitrary $ t(L). $ 
\end{abstract}
\maketitle
 \section{Introduction}

There are  some known  results concerning the relation between the center   and the derived subgroup of a group $ G. $ In 1904, Schur proved that if $ G $ is a group such that the order of $ G/Z(G) $ is finite, then the order of $ G' $ is finite. Although,  converse of the  Schur's theorem 
was not true in general, but the second author in \cite{sN} showed that $ \mid G/Z(G) \mid \leq \mid G'\mid^{d(G/Z(G))}.  $ This shows that $G/ Z(G)$ is finite if 
$G'$ is finite and $G/Z(G)$ is finitely generated. Hence
by putting some extra conditions on the group  converse of the Schur's theorem is true.
 \\
By some results due to  Lazard there are in \cite{Ber} analogies  between groups and Lie algebras, but  they are not perfect most of  times and  careful distinctions shall be done.  
Let $ L $ be a Lie algebra and   $ L^{2}, $  $ Z(L) $ and $ Z_2(L) $ are   derived subalgebra, the center and the second center of $ L. $ 
   Similar to the group theory, Moneyhun in \cite{M} has showed if $ \dim L/Z(L)=n, $ then $ \dim L^{2}\leq \frac{1}{2}n(n-1). $ 
  A natural question that arises is converse of the Schur's theorem, i.e., whether the finiteness of dimension of $ L^{2} $ implies the finiteness of dimension $ L/Z(L). $
 In \cite{A}, a similar  result of \cite{sN} for  groups shows that for a Lie algebra $ L $ we have $ \dim L/Z(L)=d \dim L^{2}-t(L)$ for an integer $ t(L)\geq 0 $ where $ d $ is  the minimal  number of generators of $ L/Z(L). $ 
 This shows that converse of the Schur's theorem is true when $d$ and $\dim L^2$ are finite.
 In this paper, we determine the structure of all  nilpotent Lie algebras $ L $  up to isomorphism for $ t(L)\in \lbrace 0, 1, 2 \rbrace. $\\
We recall that a Lie algebra $ L $ is Heisenberg provided that $ \dim L^{2}=\dim Z(L) $ and $ \dim L^{2}=1. $ Such Lie algebras are odd dimension with the following structure
\begin{equation*}
\langle x_1,y_1, \dots, x_{n-1},y_{n-1},x_n,y_n, z \mid [x_i, y_i]=z \rangle.
\end{equation*}
Here we denote an abelian Lie algebra of dimension n and the Heisenberg Lie algebra
of dimension $2m + 1$ by $A(n)$ and $H(m)$, respectively.  Also, we use the following structures from \cite{G} for all nilpotent Lie algebras 
$ L $  of dimension at most $ 6 $ with $ \dim L^{2}\geq 2$ over an arbitrary field $ \mathbb{F}. $
\[L_{4,3} = \langle x_1, x_2, x_3, x_4 \ | \ [x_1,x_2]=x_3, [x_1,x_3]=x_4 \rangle, \]
\[L_{5,3} = L_{4,3} \oplus A(1), \]
 \[L_{5,5} = \langle x_1, x_2, x_3, x_4, x_5 \ | \ [x_1,x_2]=x_3, [x_1,x_3]=[x_2,x_4]=x_5 \rangle, \]
\[L_{5,6} =  \langle x_1, x_2, x_3, x_4, x_5 \ | \ [x_1,x_2]=x_3, [x_1,x_3]=x_4, [x_1,x_4]=[x_2,,x_3]=x_5  \rangle,\]
\[L_{5,7} =  \langle x_1, x_2, x_3, x_4, x_5 \ | \ [x_1,x_2]=x_3, [x_1,x_3]=x_4, [x_1,x_4]=x_5 \rangle,\]
\[L_{5,8} =  \langle x_1, x_2, x_3, x_4, x_5 \ | \ [x_1,x_2]=x_4, [x_1,x_3]=x_5 \rangle,\]
\[L_{5,9} =  \langle x_1, x_2, x_3, x_4, x_5 \ | \ [x_1,x_2]=x_3, [x_1,x_3]=x_4, [x_2,x_3]=x_5 \rangle.\]
\[L_{6,k}=L_{5,k} \oplus A(1) ~~\text{for}~~ \ k=3, 5, 6, 7,8, 9,\]
\[ L_{6,10}= \langle x_{1}, \cdots, x_{6}  \ | \   [x_{1},x_{2}]=x_{3},[x_{1},x_{3}]=x_{6}, [x_{4},x_{5}]=x_{6}\rangle, \]
  \[ L_{6,11}= \langle x_{1}, \cdots, x_{6}  \ | \   [x_{1},x_{2}]=x_{3},[x_{1},x_{3}]=x_{4}, [x_{1},x_{4}]=x_{6}, [x_{2},x_{3}]=x_{6},[x_{2},x_{5}]=x_{6} \rangle, \]
 \[ L_{6,12}= \langle x_{1}, \cdots, x_{6}  \ | \   [x_{1},x_{2}]=x_{3},[x_{1},x_{3}]=x_{4},[x_{1},x_{4}]=x_{6}, [x_{2},x_{5}]=x_{6} \rangle, \]
   \[ L_{6,13}= \langle x_{1}, \cdots, x_{6}  \ | \   [x_{1},x_{2}]=x_{3},[x_{1},x_{3}]=x_{5}, [x_{2},x_{4}]=x_{5}, [x_{1},x_{5}]=x_{6},[x_{3},x_{4}]=x_{6} \rangle, \] 
   \[ L_{6,14}= \langle x_{1}, \cdots, x_{6}  \ | \   [x_{1},x_{2}]=x_{3},[x_{1},x_{3}]=x_{4},  [x_{1},x_{4}]=x_{5}, [x_{2},x_{3}]=x_{5},\]
   \[[x_{2},x_{5}]=x_{6},[x_{3},x_{4}]=-x_{6} \rangle, \]
   \[ L_{6,15}= \langle x_{1}, \cdots, x_{6}  \ | \   [x_{1},x_{2}]=x_{3},[x_{1},x_{3}]=x_{4}, [x_{1},x_{4}]=x_{5}, [x_{2},x_{3}]=x_{5},
   \]
   \[[x_{1},x_{5}]=x_{6},[x_{2},x_{4}]=x_{6}\rangle, \]
   \[  L_{6,16}= \langle x_{1}, \cdots, x_{6}  \ | \   [x_{1},x_{2}]=x_{3},[x_{1},x_{3}]=x_{4}, [x_{1},x_{4}]=x_{5}, [x_{2},x_{5}]=x_{6},[x_{3},x_{4}]=-x_{6} \rangle,\]
  \[ L_{6,17}= \langle x_{1}, \cdots, x_{6}  \ | \  [x_{1},x_{2}]=x_{3},[x_{1},x_{3}]=x_{4}, [x_{1},x_{4}]=x_{5}, [x_{1},x_{5}]=x_{6},[x_{2},x_{3}]=x_{6}\rangle,\]
 \[ L_{6,18}= \langle x_{1}, \cdots, x_{6}  \ | \   [x_{1},x_{2}]=x_{3},[x_{1},x_{3}]=x_{4}, [x_{1}, x_{4}]=x_{5}, [x_{1},x_{5}]=x_{6} \rangle,\]
   \[ L_{6,19}(\varepsilon) =\langle x_{1}, \cdots, x_{6}  \ | \  [x_{1},x_{2}]=x_{4}, [x_{1},x_{3}]=x_{5}, [x_{1},x_{5}]=x_{6},[x_{2},x_{4}]=x_{6}, \]
  \[ [x_{3},x_{5}]=\varepsilon x_{6}\rangle,~~\text{for all} ~~\varepsilon\in  \mathbb{F}^{*}/(\stackrel{*}{\sim})\] 
    \[ L_{6,20}= \langle x_{1}, \cdots, x_{6}  \ | \   [x_{1},x_{2}]=x_{4},[x_{1},x_{3}]=x_{5}, [x_{1},x_{5}]=x_{6}, [x_{2},x_{4}]=x_{6} \rangle ,\]
  \[ L_{6,21}(\varepsilon)= \langle x_{1}, \cdots, x_{6}  \ | \   [x_{1},x_{2}]=x_{3},[x_{1},x_{3}]=x_{4}, [x_{2},x_{3}]=x_{5}, [x_{1},x_{4}]=x_{6},\]
 \[ [x_{2},x_{5}]=\varepsilon x_{6} \rangle, ~~\text{for all} ~~\varepsilon\in  \mathbb{F}^{*}/(\stackrel{*}{\sim})\]
   \[ L_{6,22}(\varepsilon) =\langle x_{1}, \cdots, x_{6}  \ | \   [x_{1},x_{2}]=x_{5},[x_{1},x_{3}]=x_{6}, [x_{2},x_{4}]=\varepsilon x_{6},\]
   \[ [x_{3},x_{4}]=x_{5} \rangle, ~~\text{for all} ~~\varepsilon\in  \mathbb{F}/(\stackrel{*}{\sim})\]
 \[ L_{6,23}= \langle x_{1}, \cdots, x_{6}  \ | \  [x_{1},x_{2}]=x_{3},[x_{1},x_{3}]=x_{5}, [x_{1},x_{4}]=x_{6},[x_{2},x_{4}]=x_{5} \rangle,\]
\[ L_{6,24}(\varepsilon)= \langle x_{1}, \cdots, x_{6}  \ | \   [x_{1},x_{2}]=x_{3},[x_{1},x_{3}]=x_{5}, [x_{1},x_{4}]=\varepsilon x_{6}, [x_{2},x_{3}]=x_{6},\]
\[[x_{2},x_{4}]=x_{5} \rangle, ~~\text{for all} ~~\varepsilon\in  \mathbb{F}/(\stackrel{*}{\sim})\]
 \[L_{6,25}=\langle x_{1}, \cdots, x_{6}  \ | \   [x_{1},x_{2}]=x_{3},[x_{1},x_{3}]=x_{5}, [x_{1},x_{4}]=x_{6} \rangle, \]
   \[ L_{6,26}=\langle x_{1}, \cdots, x_{6}  \ | \   [x_{1},x_{2 }]=x_{4},[x_{1},x_{3}]=x_{5}, [x_{2},x_{3}]=x_{6} \rangle,\]
  \[L_{6,27}= \langle x_{1}, \cdots, x_{6} \ | \  [x_{1},x_{2}]=x_{3},[x_{1},x_{3}]=x_{5}, [x_{2},x_{4}]=x_{6} \rangle,\]
 \[ L_{6,28}= \langle x_{1}, \cdots, x_{6} \ | \  [x_{1},x_{2}]=x_{3},[x_{1},x_{3}]=x_{4}, [x_{1},x_{4}]=x_{5} , [x_{2},x_{3}]=x_{6} \rangle.\]
 \[ L^{(2)}_{6,1}=\langle  x_{1}, \cdots, x_{6} \ | \ [x_{1},x_{2}]=x_{3},[x_{1},x_{3}]=x_{5}, [x_{1},x_{5}]=x_{6} , [x_{2},x_{4}]=x_{5}+x_{6}, [x_3, x_4]=x_6 \rangle,
 \]
   \[ L^{(2)}_{6,2}=\langle  x_{1}, \cdots, x_{6} \ | \ [x_{1},x_{2}]=x_{3},[x_{1},x_{3}]=x_{4}, [x_{1},x_{4}]=x_{5} , [x_{1},x_{5}]=x_{6}, [x_2, x_3]=x_5+x_6,\]
\[  [x_2, x_4]=x_6 \rangle,
 \]
 \[ L^{(2)}_{6,3}(\varepsilon)=\langle  x_{1}, \cdots, x_{6} \ | \ [x_{1},x_{2}]=x_{3},[x_{1},x_{3}]=x_{4}, [x_{1},x_{4}]=x_{5} ,  [x_2, x_3]=x_5+\varepsilon x_6,\]
\[  [x_{2},x_{5}]=x_{6}, [x_2, x_4]=x_6 \rangle ~~ \text{where}~~ \varepsilon \in \mathbb{F}^{*}/(\stackrel{*+}{\sim})
 \]
 \[ L^{(2)}_{6,4}(\varepsilon)=\langle  x_{1}, \cdots, x_{6} \ | \ [x_{1},x_{2}]=x_{3},[x_{1},x_{3}]=x_{4}, [x_{1},x_{4}]=x_{5} ,  [x_2, x_3]=\varepsilon x_6,\]
\[  [x_{2},x_{5}]=x_{6}, [x_3, x_4]=x_6 \rangle ~~ \text{where}~~ \varepsilon \in \mathbb{F}^{*}/(\stackrel{*+}{\sim})
 \]
 \[ L^{(2)}_{6,5}=\langle  x_{1}, \cdots, x_{6} \ | \ [x_{1},x_{2}]=x_{4},[x_{1},x_{3}]=x_{5}, [x_{2},x_{5}]=x_{6} ,  [x_3, x_4]=x_6 \rangle, \]
\[ L^{(2)}_{6,6}=\langle  x_{1}, \cdots, x_{6} \ | \ [x_{1},x_{2}]=x_{3},[x_{1},x_{3}]=x_{4}, [x_{1},x_{5}]=x_{6} ,  [x_2, x_3]=x_5, [x_2, x_4]=x_6 \rangle,\]
\[ L^{(2)}_{6,7}(\eta)=\langle  x_{1}, \cdots, x_{6} \ | \ [x_{1},x_{2}]=x_{5},[x_{1},x_{3}]=x_{6}, [x_{2},x_{4}]=\eta x_{6} ,  [x_3, x_4]=x_5+x_6 \rangle \]
\[ \text{where}~~  \eta \in \lbrace 0, \omega \rbrace,\]
\[ L^{(2)}_{6,8}(\eta)=\langle  x_{1}, \cdots, x_{6} \ | \ [x_{1},x_{2}]=x_{3},[x_{1},x_{3}]=x_{5}, [x_{1},x_{4}]=\eta x_{6} ,  [x_2, x_3]=x_6, [x_2, x_4]=x_5+x_6 \rangle \]
\[ \text{where}~~  \eta \in \lbrace 0, \omega \rbrace,\]
The following lemma and theorem are useful in the next section.
\begin{lem}\cite[Lemma 3.3]{N1}\label{l0}
Let $ L $ be an $n $-dimensional Lie algebra and $ \dim L^{2}=1. $ Then for some $ m\geq 1 $
\begin{equation*}
L\cong H(m)\oplus A(n-2m-1).
\end{equation*}
\end{lem}
\begin{prop}\cite[Proposition 2.1]{J}\label{prop0}
Let $ L $ be a $ d $-generator nilpotent Lie algebra of dimension $ n $ with the derived subalgebra of dimension $ m. $ Then 
\begin{equation*}
d=\dim L/L^{2}=n-m. 
\end{equation*}
\end{prop}
In the following table for convenient of the reader we get $ \dim L/Z(L), $ $ d( L/Z(L) )$ and $ \dim L^{2} $ for all nilpotent Lie algebras $ L $ of dimension at most $ 6 $ with $ \dim L^{2}\geq 2. $
\begin{longtable}{ccccc}
\multicolumn{3}{c}{\textbf{Table 1. } }\\
\hline \multicolumn{1}{c}{\textbf{Name}} & \multicolumn{1}{c}{$ \dim L/Z(L) $}   & \multicolumn{1}{c}{$ d( L/Z(L) )$} & \multicolumn{1}{c}{$ \dim L^{2} $} \\
\hline
\endhead
\hline \multicolumn{2}{r}{\small \itshape continued on the next page}
\endfoot
\endlastfoot
$L_{4,3},$ $L_{5,3}$ $ L_{6,3} $ & 3& 2 & 2\\
\\
$L_{5,5},$ $L_{6,5}$& 4& 3 & 2\\
\\
$L_{5,6},$ $L_{5,7},$ $L_{6,6},$ $L_{6,7}$& 4& 2 & 3\\
\\
$L_{5,8},$ $L_{6,8}$ & 3& 3 & 2\\
\\
$L_{5,9},$ $L_{6,9}$   & 3& 2 & 3 \\
\\
$L_{6,10}$& 5& 4 & 2 \\ 
\\
$L_{6,11},$ $L_{6,12},$ $L_{6,13},$ $L_{6,19}(\varepsilon),$ $L_{6,20}$& 5& 3 & 3 \\ 
\\
$L^{(2)}_{6,1},$ $L^{(2)}_{6,5}$ & 5& 3 & 3 \\ 
\\
$L_{6,14},$ $L_{6,15},$ $L_{6,16},$ $L_{6,17},$ $L_{6,18},$ $L_{6,21}(\varepsilon)$& 5& 2 & 4 \\
\\
$L^{(2)}_{6,2},$ $L^{(2)}_{6,3},$ $L^{(2)}_{6,4},$ $L^{(2)}_{6,6}$ & 5& 2 & 4 \\
\\
$L_{6,22}(\varepsilon),$ $L^{(2)}_{6,7}(\eta)$& 4& 4 & 2 \\ 
\\
$L_{6,23},$ $L_{6,24}(\varepsilon),$ $L_{6,25}$  & 4& 3 & 3 \\ 
\\
$L_{6,27},$ $L^{(2)}_{6,8}(\eta)$ & 4& 3 & 3 \\
\\
$L_{6,26}$ & 3& 3 & 3 \\
\\ 
$L_{6,28}$ & 4& 2 & 4\\
\\ 
\hline
 \label{ta1}
\end{longtable}
Finally, we recall the concept
a capable Lie algebra 
that be used in the next section. 
A Lie algebra $ L $ is called capable if $ L\cong H/Z(H) $ for a Lie algebra $ H $. 
\section{Main Results}
 In this section, we obtain the structure of all nilpotent Lie algebras $ L $ up to isomorphism when $ t(L)\in \lbrace 0, 1, 2 \rbrace. $
 The following  two lemmas show that it is sufficient to obtain 
the structure of all  nilpotent stem Lie algebras $ T $ such that $ \dim T/Z(T)=d(T/Z(T)) \dim T^{2}-t(L)$ for all $ t(L)\in \lbrace 0, 1, 2 \rbrace. $ 
\begin{lem}\cite[Proposition 3.1]{J1}\label{L1}
 Let $L$ be a of finite dimensional Lie algebra. Then $L = T \oplus A$ and $Z(T) = L^{2} \cap Z(L),$ where $A$ is abelian and $T$ is non-abelian Lie algebra.
\end{lem}
\begin{lem}\label{l1}
Let $ L $ be an $ n $-dimensional nilpotent Lie algebra. Then   $ L=T\oplus A(k) $ where $ T $ is a stem Lie algebra and $ A(k) $ is an abelian Lie algebra  for all $ k\geq 0 $ and  $ \dim L/Z(L)=d(L/Z(L)) \dim L^{2}-t(L) $ for all $ t(L)\geq 0 $ if and only if $ \dim T/Z(T)=d(T/Z(T)) \dim T^{2}-t(L)$ for all $ t(L)\geq 0. $ 
\end{lem}
\begin{proof}
By using Lemma \ref{L1}, we have $ L=T\oplus A(k) $ for some $ k\geq 0$. Clearly  $ \dim L/Z(L)=\dim T/Z(T)$ and so  $ \dim (L/Z(L))^{2}=\dim (T/Z(T))^{2}. $  Hence 
$ d(L/Z(L))\\=d(T/Z(T))$ by using Proposition \ref{prop0}. On the other hand, $ \dim L^{2}=\dim T^{2}. $ Therefore 
the result follows.
\end{proof}
\begin{prop}\label{prop1}
Let $ L $ be an $ n $-dimensional nilpotent Lie algebras and $ L/Z(L) $ be generated  by $ d $ elements such that $ \dim L^{2}\geq 2. $ Then $ t(L)> 0. $
\end{prop}
\begin{proof}
Let $ L $ be an $ n $-dimensional nilpotent Lie algebras such that $ \dim L^{2}\geq 2. $ We claim that $ t(L)\geq 1. $  
Let $ L/Z(L) $ be an  abelian Lie algebra, then $ \dim L/Z(L)=d $ and so
\begin{align}\label{eq2.1}
\dim L/Z(L)&=d\dim L^{2}-d\dim L^{2}+\dim L/Z(L) \cr
&= d\dim L^{2}+ \dim L/Z(L) (1-\dim L^{2})
\end{align} 
Since $ \dim L^{2}\geq 2, $   we have $\dim L/Z(L) \leq d\dim L^{2}-1$ by using \eqref{eq2.1} and so $ t(L)\geq 1. $\\
 We know that 
 $$ d((L/Z(L))/Z(L/Z(L)))= d(L/Z_{2}(L))$$
  and so 
\begin{align}\label{eq4}
d((L/Z(L))/Z(L/Z(L)))&=d(L/Z_{2}(L))\cr
&=\dim (L/Z_{2}(L))- \dim (L/Z_{2}(L))^{2}\cr
&=\dim (L/Z_{2}(L))- \dim (L^{2}+Z_2(L))/Z_2(L)\cr
&=\dim (L/Z(L))- \dim (L^{2}+Z_2(L))/Z(L)
\end{align}
 On the other hand,  
 \begin{align*}
 &\dim (L^{2}+Z_2(L))/Z(L)\cr
 &=\dim (L^{2}+Z_2(L))/ (L^{2}+Z(L)) +\dim (L^{2}+Z(L))/Z(L). 
   \end{align*}
   thus
 \begin{align*}
d((L/Z(L))/Z(L/Z(L)))&=d- \dim (L^{2}+Z_2(L))/ (L^{2}+Z(L)) 
\end{align*}
by using \eqref{eq4}.
Hence 
\begin{align}\label{e6}
& d(L/Z(L))/Z(L/Z(L)) \dim (L/Z(L))^{2}+\dim Z_2(L)/Z(L)\cr
&=d\dim L^{2}-d \dim L^{2}\cap Z(L)- \dim (L^{2}+Z_2(L))/ (L^{2}+Z(L)) \dim(L/Z(L))^{2}\cr
&+\dim Z_2(L)/Z(L).\cr
\end{align}
Also, $ d=\dim L/(L^{2}+Z(L)) $ and $  \dim L^{2}\cap Z(L)\geq 1, $ hence
\begin{align*}
& d(L/Z(L))/Z(L/Z(L)) \dim (L/Z(L))^{2}+\dim Z_2(L)/Z(L)\cr
&=d\dim L^{2}-d (\dim L^{2}\cap Z(L)-1)-\dim L/(L^{2}+Z(L))\cr
&- \dim (L^{2}+Z_2(L))/ (L^{2}+Z(L)) \dim(L/Z(L))^{2}+\dim Z_2(L)/Z(L).
\end{align*}
Since
\begin{align*}
& \dim Z_2(L)/Z(L) - \dim (L^{2}+Z_2(L))/ (L^{2}+Z(L)) \dim(L/Z(L))^{2}-\dim L/(L^{2}+Z(L))\cr
&\leq  \dim Z_2(L)/Z(L)-\dim(L/Z(L))^{2}-\dim L/(L^{2}+Z(L))\cr
&=\dim Z_2(L)/Z(L)-\dim L/Z(L)\cr
&=-\dim L/Z_2(L),
\end{align*}
we have
\begin{align}\label{eq2.5}
&d(L/Z(L))/Z(L/Z(L)) \dim (L/Z(L))^{2}+\dim Z_2(L)/Z(L)\cr
&\leq d\dim L^{2}-d (\dim L^{2}\cap Z(L)-1)-\dim L/Z_2(L).
\end{align}
Since $\dim (L/Z(L))/Z(L/Z(L))\leq d(L/Z(L))/Z(L/Z(L)) \dim (L/Z(L))^{2}, $ we have 
\begin{equation}\label{eq2.4}
\dim L/Z(L) \leq d\dim L^{2}-d (\dim L^{2}\cap Z(L)-1)-\dim L/Z_2(L).
\end{equation}
by using \eqref{eq2.5}.\\
Assume that $ \dim( L/Z(L))^{2}=1. $ Since $ \dim( L/Z(L))^{2}=1 $  and $ L/Z(L) $ is capable,  $ L/Z(L)\cong H(1)\oplus A(n-k-3) $ such that $ \dim Z(L)=k\geq 1$ by using Lemma \ref{l0} and \cite[Theorem 3.5]{P}. Hence $ \dim Z_2(L)=n-2 $ and so $ \dim L/Z_2(L)=2. $
Thus $ t(L)\geq 2 $
 by using \eqref{eq2.4}.
Now, let $ \dim (L/Z(L))^{2}\geq 2. $ We proceed by introduction on $ n $ to prove our result. Since $ \dim L^{2}\geq 2, $  $ n\geq 4. $  If $ n=4, $   then  there is no $ 4 $-dimensional nilpotent Lie algebra $ L $ such that $t(L)=0 $  by looking at the Table \ref{ta1} and $ t(L)\geq 1. $ Since  $ \dim (L/Z(L))^{2}\geq 2, $  we have 
\begin{equation}\label{e3}
 \dim (L/Z(L))/Z(L/Z(L))\leq d(L/Z(L))/Z(L/Z(L)) \dim (L/Z(L))^{2}-1
\end{equation}
by using the induction hypothesis. 
Therefore  $ t(L)\geq 1 $ by using \eqref{e3} and \eqref{eq2.5}, as required.
\end{proof}
\begin{thm}\label{th1}
Let $ L $ be an  $ n $-dimensional nilpotent Lie algebra such that  $ L/Z(L) $ be generated by $ d $ elements and  $ t(L)=0.$ Then $ L $ is isomorphic to  $ A(n) $ for all $ n\geq 1$ or  $ H(m)\oplus A(n-2m-1) $ for all $ m\geq 1. $   
\end{thm}
\begin{proof}
If $ \dim L^{2}\geq 2, $ then there is no such a nilpotent Lie algebra $ L $ by using Proposition \ref{prop1}. Hence $ \dim L^{2}\leq 1 $ and so $ L $ is isomorphic to an abelian Lie algebra $ A(n) $ for all $ n\geq 1$ or $ H(m)\oplus A(n-2m-1) $ for all $ m\geq 1$ by using Lemma \ref{l0}.
\end{proof}
\begin{prop}\label{prop2}
Let $ L $ be an $ n $-dimensional nilpotent Lie algebra, $ L/Z(L) $ be $ d $-generated and $ \dim L^{2}\geq 3. $ Then $ t(L)\geq 2. $
\end{prop}
\begin{proof}
Let $ \dim (L/Z(L))^{2}=0. $ Since $ \dim L^{2}\geq 3, $  we have $ t(L) \geq 2$ by using \eqref{eq2.1}.
If $ \dim (L/Z(L))^{2}=1, $ then $\dim L/Z_2(L)=2 $ by using a similar method used in the proof of Proposition \ref{prop1}. So 
$ t(L)\geq 2 $  by using \eqref{eq2.4}.
Now, let $ \dim (L/Z(L))^{2}=2. $ Then
\begin{align}\label{eq2.6}
 \dim (L/Z(L))/Z(L/Z(L))\leq d(L/Z(L))/Z(L/Z(L)) \dim (L/Z(L))^{2}-1
\end{align}
 by using Proposition \ref{prop1}. Also, $ \dim (L^{2}\cap Z(L))-1\geq 2 $ or $ \dim L/Z_2(L)\geq 1 $ thus   
 $ t(L)\geq 2$ by using \eqref{eq2.6} and \eqref{eq2.5}.
Let $ \dim (L/Z(L))^{2}\geq 3. $ Since $ \dim L^{2}\geq 3, $ we have $ \dim L\geq 5. $ If $ n=5, $ then there is no such a Lie algebra by looking at Table \ref{ta1}. By using the induction hypothesis,
\begin{equation}\label{le}
 \dim (L/Z(L))/Z(L/Z(L))\leq d(L/Z(L))/Z(L/Z(L)) \dim (L/Z(L))^{2}-2.
\end{equation}
  One can see  $ t(L)\geq 2 $ by using  \eqref{le} and \eqref{eq2.5}.
\end{proof}
\begin{thm}\label{th2}
Let $ L $ be an  $ n $-dimensional nilpotent Lie algebra and   $ L/Z(L) $ be generated by $ d $ elements such that   $ t(L)=1.$   Then $ L $ is isomorphic to  $ L_{4,3}\oplus A(n-4) $ for all $ n\geq 4. $
\end{thm}
\begin{proof}
If $ \dim L^{2}\leq 1, $ then $ L $ is isomorphic to $ A(n) $  for all $ n\geq 1 $ or $ H(m)\oplus A(n-2m-1) $ for all $ m\geq 1 $ and so $ t(L)=0. $ Hence there is no such a nilpotent Lie algebra with  $  t(L)=1$ when  $ \dim L^{2}\leq 1. $
Assume that $ \dim L^{2}\geq 3. $ Since $ \dim L^{2}\geq 3, $
 there is no such a nilpotent Lie algebra $ L $  with  $  t(L)=1$ by using Proposition \ref{prop2}. Let $ \dim L^{2}=2. $ Then $ L $ is of nilpotency  two or three. In the rest it is sufficient to obtain the structure of all stem nilpotent Lie algebras  $ L $ with $ t(L)\geq 2 $ by using Lemma \ref{l1}  and then we determine   the structure of  all nilpotent Lie algebras  $ L $ with $ t(L)\geq 2. $    Assume that $ L $ is a stem Lie algebra  of nilpotency class two.  Hence $ Z(L)=L^{2} $ and so   $ d(L/Z(L))=\dim L/Z(L). $ On the other hand, $  t(L)=1$ by our assumption. Thus $ n=3. $ But there is no $ 3 $-dimensional stem  nilpotent Lie algebra $ L $ such that  $  t(L)=1$ by looking at the classification of nilptent Lie algebra in \cite{G}.  Let $ L $ be a stem Lie algebra of nilpotency class $ 3. $ Then $ \dim Z(L)=1 $ and $ L/Z(L)\cong H(1)\oplus A(n-4). $ Hence $ d(L/Z(L))=n-2 $ and $ \dim L/Z(L)=n-1. $ Since  $  t(L)=1 $ by our assumption, we have $ n=4. $ By looking at Table \ref{ta1}, we have $ L\cong L_{4,3} $ when $ L $ is a stem nilpotent Lie algebra. Therefore  $ L $ is isomorphic to $ L_{4,3}\oplus A(n-4) $ for all $ n\geq 4 $ by using Lemma \ref{l1}.  
\end{proof}
\begin{prop}\label{prop3}
Let $ L $ be an $ n $-dimensional nilpotent Lie algebra, $ L/Z(L) $ be $ d $-generated and $ \dim L^{2}\geq 4. $ Then $ t(L)\geq 3. $
\end{prop}
\begin{proof}
Let $ \dim (L/Z(L))^{2}=0. $ Since $ \dim L^{2}\geq 4, $  we have $ t(L) \geq 3$ by using \eqref{eq2.1}.
Assume that  $ \dim (L/Z(L))^{2}=1. $ Since $ L/Z(L) $ is capable,  we have 
 $ L/Z(L)\cong H(1)\oplus A(n-3-k) $ 
such that $ \dim Z(L)=k\geq 1 $ by using Lemma \ref{l0} and \cite[Theorem 3.5]{P}.  Hence $ \dim L^{2}/L^{2} \cap Z(L)=1. $
Since  $ \dim L^{2}\geq 4, $ we have
$\dim L^{2} \cap Z(L) \geq 3. $ Also, $ \dim L/Z_2(L)\geq 2. $ On the other hand,   $\dim L^{2} \cap Z(L) \geq 3, $ $ \dim L/Z_2(L)\geq 2 $ and
$ \dim L/Z(L)/Z(L/Z(L))\leq d(L/Z_2(L))\dim (L/Z(L))^{2} $  thus $ t(L)\geq 3 $ by using \eqref{eq2.5}.
Assume that $ \dim (L/Z(L))^{2}= 2.$  By using  Proposition \ref{prop1}  we have
\begin{equation}\label{e4}
\dim L/Z(L)/Z(L/Z(L))\leq d(L/Z_2(L))\dim (L/Z(L))^{2}-1.
\end{equation}
If $ t(L/Z(L))=1, $ then  $ L/Z(L)\cong L_{4,3}\oplus A(n-4-k) $ such that $ \dim Z(L)=k\geq 1 $ by using Theorem \ref{th2}. 
 Since  $L/Z(L)\cong L_{4,3}\oplus A(n-4-k) $ such that $ \dim Z(L)=k\geq 1 $  and since $ \dim L^{2}\geq 4, $ we have $ \dim L^{2}\cap Z(L)\geq 2.$ Also, $ \dim L/Z_2(L)\geq 1. $ 
 Therefore $ t(L)\geq 3$  by using \eqref{e4} and \eqref{eq2.5}. If $ \dim L/Z(L) $ does not attain the upper bound \eqref{e4}, then
 \begin{equation}\label{eq2.9}
\dim L/Z(L)/Z(L/Z(L))\leq d(L/Z_2(L))\dim (L/Z(L))^{2}-2
\end{equation}
and so 
\begin{align}\label{eq2.10}
\dim L/Z(L) \leq d\dim L^{2}-d (\dim L^{2}\cap Z(L)-1)-\dim L/Z_2(L)-2
\end{align}
by using \eqref{eq2.5}.
 Since $ (\dim (L^{2}\cap Z(L))-1)\geq 3 $ or $ \dim L/Z_2(L)\geq 1, $  we have $ t(L)\geq 3. $
  If $ \dim (L/Z(L))^{2}= 3,$ then
 \begin{equation}\label{eq2.11}
 \dim (L/Z(L))/Z(L/Z(L))\leq d(L/Z(L))/Z(L/Z(L)) \dim (L/Z(L))^{2}-2
\end{equation}
by using Proposition \ref{prop2}. On the other hand, $ \dim (L^{2}\cap Z(L))-1\geq 3 $ or $ \dim L/Z_2(L)\geq 1 $ thus $ t(L)\geq 3 $ by using \eqref{eq2.11} and \eqref{eq2.5}.
  \\
Let  $ \dim (L/Z(L))^{2}\geq  4.$ Then we claim that $ t(L)\geq 3. $ 
Since   $ \dim L^{2}\geq 4, $ we have $ \dim L\geq 6. $ If $ n=6, $ then $ t(L)\geq 3 $
for of all nilpotent  Lie algebras $ L $ with $ \dim L^{2}\geq 4 $ by  looking at Table \ref{ta1}. By using the induction hypothesis,
\begin{equation}\label{eq2.12}
 \dim (L/Z(L))/Z(L/Z(L))\leq d(L/Z(L))/Z(L/Z(L)) \dim (L/Z(L))^{2}-3.
\end{equation}
 Since $ (\dim (L^{2}\cap Z(L))-1)\geq 3 $ or $ \dim L/Z_2(L)\geq 1, $ one can see  $ t(L)\geq 3 $
   by using \eqref{eq2.12} and \eqref{eq2.5}.
  \end{proof}
\begin{thm}\label{th3}
Let $ L $ be an  $ n $-dimensional nilpotent Lie algebra and   $ L/Z(L) $ be generated by $ d $ elements such that   $ t(L)=2.$   Then $ L $ is isomorphic to one of Lie algebras  $ L_{5,5}\oplus A(n-5), $ $ L_{5,6}\oplus A(n-5) $ or $ L_{5,7}\oplus A(n-5) $ for all $ n\geq 5. $
\end{thm}
\begin{proof}
If $ \dim L^{2}\geq 4, $ then there is no such a nilpotent Lie algebra by using Proposition \ref{prop3}. 
Hence $ \dim L^{2}\leq 3. $ If $ \dim L^{2}\leq 1, $ then  $ L $ is isomorphic to $ A(n) $ for $ n\geq 1 $ or $ H(m)\oplus A(n-2m-1) $ for $ m\geq 1 $  by using Lemma \ref{l0} and so $ t(L)=0 .$ Hence in this case there is no such a nilpotent Lie algebras $ L $ with $ t(L)=2. $
 Let $ \dim L^{2}=2 $ and $ t(L)=2.$ It is sufficient to obtain the structure of all  stem nilpotent Lie algebras $ L $  with $ t(L)=2 $ and then we can determine   the structure of  all nilpotent Lie algebras $ L $ with $ t(L)=2 $  by using Lemma \ref{l1}. 
Let $ L $ is of nilpotency class two and stem. Since $ t(L)=2 $ and $ L $ is of nilpotency class two, we have $ Z(L)=L^{2} $ and so $ n=4. $
Now, by looking at the Table \ref{ta1} there is no such a nilpotent Lie algebra $ L$ with $ t(L)=2. $ Let  $ L $ be of nilpotency class three and stem. Then $ \dim Z(L)=1. $ Since  $ t(L)=2 $  and $ \dim Z(L)=1, $ we have  $ \dim L=5 .$
Thus $ L $ is isomorphic to  $ L_{5,5} $ by looking at  Table \ref{ta1}.
In the case $ \dim L^{2}=3 $ and $ L $ is a stem nilpotent Lie algebra $ L, $  we can see that $ \dim L $ is equal to $ 4 $ or $ 5 $ by using a similar method and so $ L $ is isomorphic to one of the  $ L_{5,6} $ or $ L_{5,7} $ by using at Table \ref{ta1}.
Therefore $ L $ is isomorphic to one of Lie algebras  $ L_{5,5}\oplus A(n-5), $ $ L_{5,6}\oplus A(n-5) $ or $ L_{5,7}\oplus A(n-5) $ for all $ n\geq 5 $ by using Lemma \ref{l1}. 
\end{proof}
\begin{theorem}
Let $ L $ be an  $ n $-dimensional nilpotent Lie algebra such that  $ L/Z(L) $ is generated by $ d $ elements. Then 
\begin{itemize}
\item[(i).] $ t(L)=0$ if and only if  $ L $ is isomorphic to one of Lie algebras $ A(n) $ for all $ n\geq 1$ or  $ H(m)\oplus A(n-2m-1) $ for all $ m\geq 1, $   
\item[(ii).] $ t(L)=1$ if and only if $ L $ is isomorphic to $ L_{4,3}\oplus A(n-4) $ for all $ n\geq 4, $  
\item[(iii).] $ t(L)=2$  if and only if $ L $ is isomorphic to $ L_{5,5}\oplus A(n-5), $ $ L_{5,6}\oplus A(n-5) $ or $ L_{5,7}\oplus A(n-5) $ for all $ n\geq 5. $
\end{itemize}

\end{theorem} 
\begin{proof}
The  results follow by using Theorems \ref{th1}, \ref{th2} and \ref{th3}.
  The converse of theorem is obvious.
\end{proof}

\begin{cor} 
Let $ L $ be an  $ n $-dimensional nilpotent Lie algebra and  $ L/Z(L) $ be generated by $ d $ elements and   $ \dim L^{2}\geq 4. $  Then $ t(L)\geq 3. $
\end{cor}
In the rest we show that there exists at least a finite dimensional nilpotent Lie algebra $ L $ such that $ \dim L/Z(L)=d \dim L^{2}-t(L)$ for an arbitrary integer $ t(L)\geq 0. $ 
\begin{thm}
There exists at least a finite dimensional  nilpotent Lie algebra   $ L $ for an arbitrary integer $ t(L)\geq 0 $  such that $ L/Z(L) $ be generated by $ d $ elements.  
\end{thm}
\begin{proof}
If $ t(L)=0, $ then there is such a Lie algebra by using Theorem \ref{th1}.  We claim that $ L\cong \langle s, s_j\mid [s, s_i]=s_{i+1}, 1\leq j\leq t+2, 1\leq i\leq t+1 \rangle $ satisfies in our assumption for all $ t(L)\geq 1. $ We proceed by induction on $ t(L). $ If $ t(L)=1, $ then
\begin{equation*}
 L=L_{4,3}\cong  \langle s, s_j\mid [s, s_i]=s_{i+1}, 1\leq j\leq 3, 1\leq i\leq 2 \rangle
\end{equation*}
 by using Theorem \ref{th2}. By using the induction hypothesis, let 
 \begin{equation*}
 L\cong \langle s, s_j\mid [s, s_i]=s_{i+1}, 1\leq j\leq t+1, 1\leq i\leq t \rangle
 \end{equation*}
 satisfies  in our hypothesis.
 Put 
  \begin{equation*}
  H/Z(H)= \langle s+Z(H), s_j+Z(H)\mid [s, s_i]+Z(H)=s_{i+1}+Z(H), 1\leq j\leq t+1, 1\leq i\leq t \rangle.
  \end{equation*} 
 Hence 
$ \varphi : H/Z(H) \longrightarrow L=\langle s, s_j\mid [s,s_i]=s_{i+1}, 1\leq j\leq t+1, 1\leq i\leq t \rangle$ by given
\begin{align*}
& s+Z(H)\mapsto s,\cr
& s_j+Z(H)\mapsto s_j
\end{align*} 
 is an isomorphism.
Hence
 \begin{align*}
 H\cong \langle s, s_j\mid [s,s_i]=s_{i+1}, 1\leq j\leq t+2, 1\leq i\leq t+1  \rangle
 \end{align*}
and $ Z(H)=\langle s_{t+2} \rangle.  $ Clearly   $ \dim H/Z(H)=d \dim H^{2}-t(L),$ as required. 
\end{proof}
\section*{Conflict of interest }
All author declare that they have no conflicts of interest.

\end{document}